\documentclass{amsart}
\usepackage{amsmath, amssymb, amsthm,graphicx}
\usepackage{bm}
\usepackage{setspace}
\usepackage[utf8]{inputenc}
\usepackage[english]{babel}
\usepackage[normalem]{ulem}
\usepackage{soul}
\usepackage{mathrsfs}
\usepackage{verbatim}
\usepackage{tikz}
\usetikzlibrary{fit,calc,positioning,decorations.pathreplacing,matrix}
\usetikzlibrary{arrows,decorations.markings}

\usetikzlibrary{fit,calc,positioning,decorations.pathreplacing,matrix}
\usetikzlibrary{arrows,decorations.markings}

\makeatletter
\renewcommand{\@eqnnum}{%
    {\normalfont\normalcolor%
    \theequation}}
\renewcommand{\theequation}{%
    E.\arabic{equation}}
\makeatother

\newtheorem{thm}{Theorem}[section]
\newtheorem{prp}[thm]{Proposition}
\newtheorem{cor}[thm]{Corollary}
\newtheorem{lm}[thm]{Lemma}
\newtheorem{rmq}[thm]{Remark}

\theoremstyle{definition}
\newtheorem{df}[thm]{Definition}
\newtheorem{ex}[thm]{Example}
\newtheorem{conj}{Conjecture}

\newcommand\NN{\mathbb{N}}

\newcommand{\bsss}{\boldsymbol{\sigma}}
\newcommand{\ssi}{\sigma_{i}}
\newcommand{\ssj}{\sigma_{j}}

\newcommand{\bbw}[1]{\boldsymbol{#1}}

\newcommand{\hpiix}{\pi^*_{\scriptscriptstyle I,X}}
\newcommand{\piix}{\pi_{\scriptscriptstyle I,X}}
\newcommand{\KKI}{\kappa_I}

\title{On Parabolic subgroups of Artin-Tits groups}
\author{Eddy Godelle}
\begin{document}
\begin{abstract} We address the conjecture which states that an intersection of parabolic subgroups of  an Artin-Tits group  is a parabolic subgroup.  We prove that the conjecture is equivalent to a, a priori, weaker conjecture.  We also prove the conjecture in a specific case.  Along the way,  we provide  short and almost self-contain algebraical proofs of several classical results on Artin-Tits groups,  such as those of Van der Lek on intersection of  standard parabolic subgroups.  
\end{abstract}
\maketitle
\section*{Introduction}

We fix a finite simplicial labelled graph~$\Gamma = (I,E,m)$  with vertex set~$I$,  edge set~$E$ and label map~$m:~E\to \NN_{\geq 2}$.   We also fix two sets
$\Sigma_I = \{\ssi\mid i\in I\}$ and $S_I = \{s_i\mid i\in I\}$ that are  indexed by $I$.   If $X$ is a subset of $I$,  by $\Gamma_X = (X,E_X,m_X)$ we denote the full (labelled) subgraph of $\Gamma$ spanned by~$X$.   We set $\Sigma_X = \{\ssi \in \Sigma_I \mid i\in X\}$ and $S_X = \{s_i\in S_I \mid i\in X\}$.  If $S$ is a set,  by $S^*$ we denote the monoid of words on $S$.  The length of a word $w$ on $S$ is denoted by $\ell_S(w)$.  For convenience,  for $\{i,j\}$ in $E$,  we will often  write $m(i,j)$ for $m(\{i,j\})$.  The  Artin-Tits group~$A_I$ generated by $\Sigma_I$ and associated with $\Gamma$ is defined by the following  presentation of group:  
  \begin{equation} A_I= \biggl\langle \Sigma_I\  \biggl | \    \underbrace{\ssi\ssj\ssi\cdots}_{m(e) \textrm{ terms}}  = \underbrace{\ssj\ssi\ssj\cdots}_{m(e) \textrm{ terms}} \textrm{ for } e = \{i,j\}\in E  \biggr\rangle 
\end{equation}
The  Coxeter group  $W_I$ generated by $S_I$ and associated with $\Gamma$ is  defined by the following  presentation: 
 \begin{equation}  W_I = \biggl\langle S \  \biggl | \   \begin{array}{ll} s^2 = 1 &\textrm{ for } s\in S \\ \underbrace{s_is_js_i\cdots}_{m(e) \textrm{ terms}}  = \underbrace{s_js_is_j\cdots}_{m(e) \textrm{ terms}} & \textrm{ for } e = \{i,j\}\in E \end{array}  \biggr\rangle \end{equation} 
 
To speak about the relations that appear in the presentation of $A_I$ or to their corresponding relations in the presentation of $W_I$,  we  will speak of  the \emph{braid relations} of the presentation.  
\begin{ex} \label{exe:base:1} Consider $I = \{a,b,c\}$ and $\Gamma_I$ as below.  Then,  $A_I$ is the braid group on four strands. The groupe $W_I$ is the permutation group on $4$ elements where $S_I$ consists on the  elementary transpositions. \\   

\begin{minipage}[c]{0.50\linewidth} 
\begin{center}$W_I = \biggl\langle s_a, s_b, s_c \  \biggl | \
 \begin{array}{l} 
 s_a^2 = s_b^2  =  s_c^2  = 1 \\
 s_as_bs_a = s_bs_as_b\\ 
 s_bs_cs_b = s_cs_bs_c\\ 
 s_as_c = s_cs_a \end{array}  
 \biggr\rangle$  
 \end{center}
 \end{minipage}
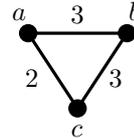
 \begin{minipage}[c]{0.45\linewidth} 
 \begin{center}
\begin{tikzpicture}[decoration={brace}][scale=2]
\draw (2.64,-1.3) node {$c$};
\draw[very thick,fill=black] (2,0) circle (.1cm);
\draw[very thick,fill=black] (3.3,0) circle (.1cm);
\draw[very thick,fill=black] (2.65,-1) circle (.1cm);
\draw[very thick] (2,0) -- +(1.3,0);
\draw[very thick] (2,0) -- +(0.65,-1);
\draw[very thick] (3.3,0) -- +(-0.65,-1);
\draw (2.65,0.25) node {$3$};
\draw (3.15,-0.6) node {$3$};
\draw (2.05,-0.6) node {$2$};
\draw (1.88,0.25) node {$a$};
\draw (3.40,0.3) node {$b$};
\end{tikzpicture}
\end{center}
\end{minipage}

\end{ex}

Note that $\Gamma$ is not the classical Coxeter graph associated with the above presentation (here no edge means no relation, and a commutation relation corresponds to an edge labelled with $2$).  
If $w$ is a word on $S_I$ or on $\Sigma_I\cup\Sigma_I^{-1}$,  by $\bbw{w}$  (in bold) we denote the corresponding element in $W_I$, or $A_I$,  respectively.   We will simply write $\ell_I$ for both $\ell_{S_I}$ and $\ell_{\Sigma_I\cup \Sigma_I^{-1}}$.  If $w'$ is another word on the same set,  we write $w\equiv_I  w'$ when $\bbw{w} = \bbw{w}'$ in the corresponding group.  Let $\theta^*_I : (\Sigma_I\cup \Sigma_I^{-1})^*  \to S_I^*$  be the morphism of monoids  that sends both  $\sigma_i$ and $\sigma_i^{-1}$ onto  $s_i$.  Clearly,  the morphism $\theta^*_I$ induces a morphism of groups $\theta_I : A_I \to W_I$. The kernel of $\theta_I$ is called the colored Artin-Tits group and will be denoted by $CA_I$.  The length $\ell_I(\bbw{w})$ of $\bbw{w}$ in $W_I$  is defined as $\ell_I(\bbw{w}) = \min\{\ell_I(v) \mid v\in S_I^* \textrm{ with } \bbw{v} = \bbw{w} \}$.  When $\ell_I(w) = \ell_I(\bbw{w})$,  we say that $w$ is a \emph{reduced} word.  If $X\subseteq I$,  it is immediate that the inclusion maps $\Sigma_X\to \Sigma$ and $S_X \to S$ extend to morphisms of groups $\iota_X : A_X\to A_I$ and $\tau_X: W_X\to W_I$.   This is well-known that these morphisms are into~\cite{VdL}.   But we are going here to provide a new direct and short algebraical proof that $\iota_X$ is into.  So,  in the sequel we identify $W_X$ with its image  by $\tau_X$ in $W_I$.  But,  at this stage,  we do not consider as known that~$\iota_X$ is into.   We denote by $A^+_I$  the submonoid of $A_I$ generated by $S_I$.  This is known by~\cite{Par2002} that $A^+_I$ and $A_I$,   possess the same presentation,  but considered as a presentation of monoid for the former, and considered as a presentation of group for the latter.  However,  we do not need to assume this result known when proving our results.  We define $\theta^*_X$,  $\theta_X$,  $A^+_X,  \ell_X$ and $\equiv_X$ similarly to $\theta^*_I$,  $\theta_{I}$,  $A^+_I,  \ell_I$ and $\equiv_I$,  respectively.   Our first objective it to  prove Proposition~\ref{main:prop2}.    This proposition claims the existence of  two maps $\hpiix$  and $\piix$  and states several of their properties.  We will see (Proposition~\ref{prp:lien:paris} and Corollary~\ref{cor:lien:paris})  that these two maps  are indeed the map $\hat{\pi}_X$ and $\pi_X$ defined  in~\cite{BlPa2022} (see also \cite{ChP2014,GoPa2012}).

\begin{prp}\label{main:prop2}  Let $X\subseteq I$.  There exists a map  $\hpiix: (\Sigma_I \cup \Sigma_I^{-1})^* \to (\Sigma_X \cup \Sigma_X^{-1})^*$ such that 
\begin{enumerate}
\item  For any word $\omega$ on $\Sigma_I \cup \Sigma_I^{-1}$,  we have $\ell_X(\hpiix(\omega)) \leq \ell_I(\omega)$. The equality holds  if and only if $\omega$ is a word on $\Sigma_X \cup \Sigma_X^{-1}$.   In  the latter case,  $\hpiix(\omega) = \omega$. 
\item if $\omega$ and $\omega'$  are words on $\Sigma_I \cup \Sigma_I^{-1}$,  then $\hpiix(\omega)$ is a prefix of $\hpiix(\omega\omega')$. 
\item \cite{BlPa2022} The map $\hpiix$  induces a map $\piix : A_I \to A_X$.
\item  If $\bbw{\omega}$ lies in $A_I$  then  $\bigg(\theta_X(\piix(\bbw{\omega}))\bigg)^{-1}\theta_I(\bbw{\omega})$ is $(X,\emptyset)$-reduced and $$\ell_I(\theta_I(\bbw{\omega})) =\ell_I\bigl(\theta_X(\piix(\bbw{\omega}))\bigr) +  \ell_I\bigg(\big(\theta_X(\piix(\bbw{\omega}))\big)^{-1}\theta_I(\bbw{\omega})\bigg)$$
\item \cite{BlPa2022}  The set map $\piix : A_I \to A_X$ restricts to an homomorphism~$\piix : CA_I \to CA_X$.
\item We have $\piix(A_I^+) = A_X^+$. 
\item For any word $\omega$ on $\Sigma_X \cup \Sigma_X^{-1}$  and any $\bbw{\omega}'$ in $A_I$,  one has  $\piix(\bbw{\omega}\bbw{\omega}') = \piix(\bbw{\omega})\,\piix(\bbw{\omega}')$.
\item   If $\bbw{\omega}$ lies in $CA_I$  then for any $\bbw{\omega}'$ in $A_I$,  one has  $\piix(\bbw{\omega}\bbw{\omega}') = \piix(\bbw{\omega})\piix(\bbw{\omega}')$.
\item If $Y\subseteq I$ and $\omega$ is a word on $\Sigma_Y \cup \Sigma_Y^{-1}$,  then  $\hpiix (\omega)$ is a word on  $\Sigma_{Y\cap X} \cup \Sigma_{Y\cap X}^{-1}$ and the restriction of  $\hpiix$ to  $(\Sigma_{Y} \cup \Sigma_{Y}^{-1})^*$  is $\pi^*_{{\scriptscriptstyle Y,Y\cap X}}$.  
\item Let $X, Y\subseteq I$ and $\omega$ be a word on $\Sigma_I \cup \Sigma_I^{-1}$,  then  $\pi^*_{I,Y}(\pi^*_{I,X}(\omega))
 = \pi^*_{I,X}(\pi^*_{I,Y}(\omega)) = \pi^*_{I,X\cap Y}(\omega)$
\item Let $X\subseteq Y\subseteq I$ and $\omega$ be a word on $\Sigma_I \cup \Sigma_I^{-1}$,  then $\pi^*_{Y,X}(\pi^*_{I,Y}(\omega))=\pi^*_{I,X}(\omega)$.
 \end{enumerate}
\end{prp}
Note that,  since at this stage we do not identify $A_X$ with its image~$\iota_X(A_X)$ in $A_I$,  in Point (vii) of the above proposition we distinguish $\bbw{\omega}$ and   $\piix(\bbw{\omega})$ (see Corollary~\ref{cor:main} below).   Some points in the above proposition, such as Point~(vi), are not explicitely stated in~\cite{BlPa2022},  but  can be deduced from the  definition of~$\hat{\pi}_X$ given in~\cite{BlPa2022}.   As a consequence of Proposition~\ref{main:prop2},  we will deduce : 
\begin{prp} \label{main:prop1} Let $X$ be a subset of $I$.  
\begin{enumerate}
\item \cite{VdL}  Let $\omega,\omega'$ be in $(\Sigma_X\cup \Sigma_X^{-1})^*$,  if  $ \omega\equiv_I  \omega'$,  then  $ \omega\equiv_X  \omega'$.  In particular,  $\iota_X$ is into.  So one can identify $A_X$ with its image in $A_I$.
\item \cite{VdL} If $Y$ is another subset of $I$ then $A_X\cap A_Y = A_{X\cap Y}$ in $A_I$.  
\item We have $A_X^+ = A_I^+\cap A_X$.
\item \cite{ChP2014} The subgroup~$A_X$ is convex : if $\bbw{\omega}$ is in $A_X$,  any of its representative words of minimal length on $\Sigma_I\cup \Sigma_I^{-1}$ is actually a word on $\Sigma_X\cup \Sigma_X^{-1}$. 
\end{enumerate}
\end{prp}
The statements in Proposition~\ref{main:prop1}  are already known.  Howewer previous proofs were long and based on topological arguments.  The interest of the present proof is that it is is short, only uses algebraical arguments and is only based on elementary prerequisites.  Indeed we only  refer to  the first pages of  \cite[Chapter 4]{Bou1968},  were Coxeter groups are defined,  and to a result of \cite{Tit1961},  which is a easy consequence of the same pages in \cite[Chapter 4]{Bou1968}.   With Proposition~\ref{main:prop1}(i)  at hand,  we deduce from Proposition~\ref{main:prop2} that 
\begin{cor}\label{cor:main}
Let $X\subseteq I$ and  identify $A_X$ with its image in $A_I$.
\begin{enumerate}
\item \cite{ChP2014}  The morphism $\piix: A_I\to A_X $ is a retraction.  
\item For any $\bbw{\omega}$ in $A_X$  and any $\bbw{\omega}'$ in $A_I$,  one has  $\piix(\bbw{\omega}\bbw{\omega}') =\bbw{\omega}\,\piix(\bbw{\omega}')$.
\item Let $X, Y\subseteq I$.  For any $\bbw{\omega}$ in $A_I$,  $\pi_{I,Y}(\pi_{I,X}(\bbw{\omega}))
 = \pi_{I,X}(\pi_{I,Y}(\bbw{\omega})) = \pi_{I,X\cap Y}(\bbw{\omega})$.
\item Let $X\subseteq Y\subseteq I$.  
\begin{enumerate} \item The restriction of  $\pi_{I,X}$ to $A_Y$ is $\pi_{Y,X}$;  \item For any $\bbw{\omega}$ in $A_I$,   $\pi_{Y,X}(\pi_{I,Y}(\bbw{\omega}))=\pi_{I,X}(\bbw{\omega})$.
\end{enumerate}
\end{enumerate}
\end{cor}
In  Point~(iii) above,  the equalities can be written  as $\pi_{X,X\cap Y}(\pi_{I,X}(\bbw{\omega}))
 = \pi_{Y,Y\cap X}(\pi_{I,Y}(\bbw{\omega})) = \pi_{I,X\cap Y}(\bbw{\omega})$ by (iv)(a). 
 
 If $X$ is a subset of $I$,  then the subgroups of $A_I$ and $W_I$ generated by $\Sigma_X$ and $S_X$,  respectively,  are called \emph{ standard parabolic subgroups}.  A \emph{parabolic subgroup} of $A_I$ is a subgroup that is conjugated to one of its standard parabolic subgroups.  In the framework of Artin-Tits groups,  a main open conjecture states that the family of parabolic subgroups of an Artin-Tits group is closed under intersection.  This conjecture can be formulated as it follows: 
\begin{conj} \label{conj:conj1} Let  $A_I$ be an Artin-Tits group,  $X,Y$  be  in $I$ and $\bbw{\omega}$ be  in $A_I$.  Then,  $(\bbw{\omega}A_{Y} \bbw{\omega}^{-1})\cap A_X$ is a parabolic subgroup of $A_I$.     
\end{conj}

By \cite{BlPa2022},   when the conjecture holds,  then  $(\bbw{\omega}A_{Y} \bbw{\omega}^{-1})\cap A_X$ is a parabolic subgroup of  $\bbw{\omega}A_{Y} \bbw{\omega}^{-1}$ and a parabolic subgroup of $A_X$.   Conjecture \ref{conj:conj1} has been proved to hold when one restricts to some particular families such has the family of Artin-Tits groups of  spherical type  \cite{Cum2019},  of FC type~\cite{Mor2021} or of large type~\cite{Cum2020}.  However, the question in the general case remains open.  Our second objective is to prove that Conjecture~\ref{conj:conj1} is equivalent to the following conjecture:
\begin{conj} \label{conj:conj2} Let  $A_I$ be an Artin-Tits group.  Let $X$ be in $I$ and $\bbw{\omega}$ be in $CA_I$.   Then $(\bbw{\omega}A_{X} \bbw{\omega}^{-1})\cap A_X$ is a  parabolic subgroup of $A_I$.        
\end{conj}

Clearly,  Conjecture~\ref{conj:conj1} implies Conjecture~\ref{conj:conj2},  so we address the other implication.   
\begin{thm} \label{main:thm1}  
 Conjecture~\ref{conj:conj2} implies Conjecture~\ref{conj:conj1}. 
 \end{thm}
 In order to prove Theorem~\ref{main:thm1}   we need to prove Conjecture~\ref{conj:conj1} in the case of specific elements $\bbw{\omega}$ of~$A_I$.  The morphism $\theta_I : A_I\to W_I$ possesses a well-defined and well-known  set section $\KKI : W_I \to A_I$ (see  Section~1).  We  will prove that
 
\begin{thm} \label{main:thm2}   Let  $X,Y$ be in $I$ and $\bbw{w}$ be in $W_I$.  Set $\bbw{\omega} = \KKI(\bbw{w})$.   Then,  there exists $\bbw{w}_1$ in $W_{X}$,  $Y_1\subseteq Y$,  and $X_1 \subseteq X$ so that $(\bbw{\omega}A_{Y} \bbw{\omega}^{-1})\cap A_X  = (\bbw{\omega}A_{Y_1} \bbw{\omega}^{-1}) = \KKI(\bbw{w_1})A_{X_1}\KKI(\bbw{w_1})^{-1}$.      
\end{thm}
Section~$1$ is devoted to the necessary backgrounds on Coxeter  groups.    In Section~2, we turn to the proof of Propositions~\ref{main:prop2}  and~\ref{main:prop1}.  We also prove that the retraction~$\piix$ is the same as the retraction~$\pi_X$ defined in~\cite{BlPa2022}.   Finally, in Section~3 we prove Theorems~\ref{main:thm1} and~\ref{main:thm2}.

\section{Background on Coxeter groups}   
We start with the elementary properties on Coxeter groups that we shall need.  They  can all be found in \cite[chap.  IV,\S 1]{Bou1968},  except Proposition~\ref{prop:base2} that can be found in \cite{Tit1961} and which is a direct consequence of  Proposition~\ref{prop:dbcoset} below.

\begin{df} \label{def:reflecset} Let $w = s_{i_1}\cdots s_{i_k}$ be a word on $S$.       
\begin{enumerate}
\item By $r_j(w)$,  or $r_i$ when no confusion is possible,  we denote the word $ s_{i_1}s_{i_2}\cdots s_{i_{j-1}}s_{i_j}s_{i_{j-1}}\cdots s_{i_2}s_{i_1}$.  We set $R(w) = (r_1(w), r_2(w),\ldots, r_k(w))$.  
\item By  $\boldsymbol{R}(w)$ we denote the associated sequence $(\boldsymbol{r_1}(w), \boldsymbol{r_2}(w),\ldots, \boldsymbol{r_k}(w))$ of elements of $W_I$.
\item By $\boldsymbol{N}(w)$ we denote  the set of elements of $W_I$ that appear an odd number of times  in $\boldsymbol{R}(w)$.  
\end{enumerate}
\end{df}
\begin{ex}\label{exe:cox1} Consider Example~\ref{exe:base:1}  and  the word $w  = s_as_bs_cs_as_cs_b$.  Then  $\boldsymbol{N}(w) =  \{ \boldsymbol{s_b} ; \boldsymbol{s_as_bs_a}\}$  because  $R(s) = \{s_a ; s_as_bs_a; s_as_bs_cs_bs_a; s_as_bs_cs_as_cs_bs_a; s_as_bs_cs_as_cs_as_cs_bs_a; s_as_bs_cs_as_cs_bs_cs_as_cs_bs_a  \}$ with  $ \bbw{s_as_bs_cs_as_cs_bs_cs_as_cs_bs_a}  = \boldsymbol{s_a}$ and $\boldsymbol{s_as_bs_cs_as_cs_as_cs_bs_a} = \boldsymbol{s_as_bs_cs_bs_a}$,  and $ \boldsymbol{s_as_bs_cs_as_cs_bs_a}  = \bbw{s_b}$.
\end{ex}
\begin{rmq}\label{rem:conca} It follows from  Definition~\ref{def:reflecset} that for any two words $w,w'$ on $S$,  the sequence $\bbw{R}(ww')$ is the concatenation of the two sequences~$\bbw{R}(ww')$ and $\bbw{w}\,\bbw{R}(w')\,\bbw{w}^{-1}$.  
\end{rmq}

\begin{prp} \cite[chap.  IV \S 1]{Bou1968} \label{prop:fourretout1}
Let $w$ and $w'$ be two words on $S_I$. 
\begin{enumerate}
\item  If  $w$ and $w'$  represent the same element in $W$,  then $\boldsymbol{N}(w) = \boldsymbol{N}(w')$. 
\item  The word  $w$  is reduced  if and only if  the cardinality of $\boldsymbol{N}(w)$ is equal to the length of $w$,  in other words when all the elements of $\boldsymbol{R}(w)$ are distinct. 
\end{enumerate}
\end{prp}
\begin{ex} Consider Example~\ref{exe:cox1}.  
 Since $\boldsymbol{N}(w) =  \{ \boldsymbol{s_b} ; \boldsymbol{s_as_bs_a}\}$,  the length on $S$ of $\bbw{w}$ is $2$.  Indeed,  $\bbw{w}  = \bbw{s_as_bs_cs_as_cs_b} = \bbw{s_bs_a}$. 
\end{ex}
\begin{prp}  \cite[chap.  IV,\S 1]{Bou1968}  \label{prop:fourretout2}
Let $w$ and $w'$ be two words on $S_I$. 
\begin{enumerate}
\item  Let $s_{i_0}$ be in $S_I$ and write $w =  s_{i_1}\cdots s_{i_k}$.  If  $w$ is a reduced word,  but~$s_{i_0}w$ is not,  then there exists $j$ so that  $s_{i_1}\cdots s_{i_k} \equiv_I s_{i_0}s_{i_1}\cdots s_{i_{j-1}}s_{i_{j+1}}  \cdots s_{i_k}$.
\item  If $w$ and $w'$ are reduced with  $w \equiv_I w'$,  then   the word $w$ can be transformed into  the word $w'$ by using  the braid relations of the presentation of $W_I$,  only.
\item Let $w$ be a word on $S_I$  and $s,t$ be in $S_I$ so that both words $sw$ and $wt$ are reduced,  but the word $swt$ is not.  Then $sw \equiv_I   wt$.  In other words $\boldsymbol{sw} = \boldsymbol{wt}$ in $W_I$.
\end{enumerate}
\end{prp}
The above proposition implies several important properties.  First, the map $S_I \to W_I, s_i \mapsto \bbw{s_i}$ is into and we can identify $S_I$ with its image $\bbw{S_I}$ in $W_I$.   Depending on the situation,  we will write $s_i$ or  $\bbw{s_i}$.  Moreover for $X\subseteq I$ we have $S_I\cap W_X = S_X$, and for any two subsets $X,Y$ of $I$,  we have $W_X\cap W_Y = W_{X\cap Y}$.   It  also follows that the morphism $\Theta_I : A_I\to W_I$ possesses a well-defined set section~$\KKI :  W_I \to A_I$ defined in the following way : for $\bbw{w}$ in $W_I$ and  any  reduced representative word  $w =  s_{i_1}\cdots s_{i_k}$ on $S_I$ of $\bbw{w}$, we have  $\KKI(\bbw{w}) = \bsss_{i_1}\cdots \bsss_{i_k}$.    Indeed,  two  reduced words on $S$ represent the same element in $W_I$ if and only if one can be transformed into the other by using braid relations only.    As a consequence,  for any $X$ included in $I$,  the subgroup~$W_X$ is convex :  if $\bbw{w}$ is in $W_X$ and $w$ is one of its reduced representative word,  then  $w$ is a word on $S_X$.    In other words,  if the product $\bbw{u}\,\bbw{v}$ lies in $W_X$ and  $\ell_I(\bbw{u}\,\bbw{v}) = \ell_I(\bbw{u})+\ell_I(\bbw{v})$ then both $\bbw{u}$ and $\bbw{v}$ lie in $W_X$, too.   In addition,  the equality $\ell_I(\bbw{u}\bbw{v}) = \ell_I(\bbw{u})+\ell_I(\bbw{v})$ holds if and only if the equality  $\KKI(\bbw{u}\bbw{v})  = \KKI(\bbw{u})\KKI(\bbw{v})$ holds.
Finally,  we recall that for $i,j$ in $I$ and distint,  then the order of the product $\bbw{s}_i\bbw{s}_j$ in $W_I$ is $m(i,j)$,  the label of the edge $\{i,j\}$ in the graph $\Gamma$ (see~\cite[Chap~V \S 4]{Bou1968}).
\begin{prp} \cite[chap.  IV,\S 1]{Bou1968} \label{prop:dbcoset} Let $X,Y$ be included in $I$ and $\bbw{w}$ be in $W_I$.  The double-coset $W_XwW_Y$ possesses a unique element $\bbw{w}_0$ of minimal length in  $W_I$.  Moreover: \begin{enumerate}
\item  For any $\bbw{w}_1$ in $W_XwW_Y$  there exist $\bbw{x}$ in $W_X$ and $\bbw{y}$ in $W_Y$  so that $\bbw{w}_1 = \bbw{x}\;\bbw{w}_0\bbw{y}$ with $\ell_I(\bbw{w}_1) = \ell_I(\bbw{x})+ \ell_I(\bbw{w}_0) + \ell_I(\bbw{y})$.   One says that $w_0$ is $(X,Y)$-reduced.  
\item For any $\bbw{x}$ in $W_X$ and any $\bbw{y}$ in $W_Y$,  one has  $\ell_I(\bbw{x}\bbw{w_0}) = \ell_I(\bbw{x}) + \ell_I(\bbw{w_0})$ and $\ell_I(\bbw{w_0}\bbw{y}) = \ell_I(\bbw{w_0}) + \ell_I(\bbw{y})$. 
\item The element $\bbw{w}$  is  $(X,Y)$-reduced if and only if it is both $(X,\emptyset)$-reduced.  and $(\emptyset,Y)$-reduced. 
\end{enumerate}
\end{prp}

The following result corresponds to Proposition~\ref{main:prop2}, but in the context of Coxeter groups.  It was annonced in \cite{Tit1961} and has been proved in \cite{Tit1974} and \cite{Sol1976}.  This is an almost direct consequence of Proposition~\ref{prop:dbcoset}.

\begin{prp} \cite{Sol1976} \label{prop:base2} Let  $X,Y\subseteq I$ and $\bbw{w}$ be in $W_I$.  Then,  there exist~$\bbw{w}_1$ in $W_{X}$,  $Y_1\subseteq Y$, and $X_1 \subseteq X$ so that $(\bbw{w}W_{Y} \bbw{w}^{-1})\cap W_X  = (\bbw{w}W_{Y_1} \bbw{w}^{-1}) = \bbw{w}_1W_{X_1}\bbw{w}_1^{-1}$.  Moreover,  if $\bbw{w}$ is $(X,Y)$- reduced,  then $\bbw{w}_1 = 1$ and $\bbw{w}S_{Y_1}\bbw{w}^{-1} = \bbw{w}S_{Y}\bbw{w}^{-1} \cap S_X = S_{X_1} $.   
\end{prp}
\begin{rmq} In \cite{Sol1976},  the author considered the case $(X,Y)$- reduced,  only.   If $\bbw{w}$ is not $(X,Y)$-reduced,  then 
 we can write $\bbw{w}= \bbw{w_1}\bbw{w'}\bbw{w_2}$ with $\bbw{w_1}$ in $W_X$,  $\bbw{w_2}$ in $W_Y$ and $\bbw{w'}$ that is $(X,Y)$-reduced.  Then $(\bbw{w}W_{Y} \bbw{w}^{-1})\cap W_X  = \bbw{w_1} (\bbw{w'}W_Y\bbw{w'}^{-1}\cap W_X)\bbw{w_1}^{-1}$. 
\end{rmq}

\begin{cor} \label{cor:base2}Let $X\subseteq I$,  $i,j$ be in $I$ and distinct and $\bbw{w}$ be in $W$.  There are only three possibilities:  $\bbw{w}W_{\{i,j\}}\bbw{w}^{-1}\cap W_X$  is trivial,  or $\bbw{w}W_{\{i,j\}}\bbw{w}^{-1}$ is included in $W_X$, or $\bbw{w}W_{\{i,j\}}\bbw{w}^{-1}\cap W_X$ contains only one not trivial element. 
\end{cor}
\begin{rmq}\label{rem:base2} In the above corollary consider the case $\bbw{w}W_{\{i,j\}}\bbw{w}^{-1}$ is included in $W_X$.  Let  $X_1 = \{i',  j'\}$  be as in Proposition~\ref{prop:base2}.  Then,  $m(i',j') = m(i,j)$ because $m(i,j)$ and $m(i',j')$ are  the orders of $s_is_j$ and  $s_{i'}s_{j'}$,  respectively,  as recalled above.    
\end{rmq}
\begin{lm} \label{lem:base1} Let $X\subseteq I$.  Let $\bbw{u}$ be in $W_I$ and $s_i$ be in $S_I$.    Write $\bbw{u} = \bbw{v}\;\bbw{w}$ such that $\bbw{v}$ lies in $W_{X}$ and $\bbw{w}$  is $(X,\emptyset)$-reduced.  Then $$\bbw{u}\;\bbw{s_i}\;\bbw{u}^{-1}\in W_{X} \iff  \bbw{w}\;\bbw{s_i}\;\bbw{w}^{-1}\in S_X \iff \bbw{w}\;\bbw{s_i} \textrm{ is not } (X,\emptyset)\textrm{-reduced}$$ Moreover,  in this case,  $\bbw{w}$ is $(X,\{i\})$-reduced.  
\end{lm}
\begin{proof}  The first equivalence is clear.  
 Implication ``\emph{$\bbw{w}\;\bbw{s}_i\;\bbw{w}^{-1}\in S_X \Rightarrow \bbw{w}\;\bbw{s}_i$ is not $(X,\emptyset)$-reduced}''  follows  from Proposition~\ref{prop:base2}, considering $Y = \{i\}$.   Assume $\bbw{w}\,\bbw{s}_i$ is not  $(X,\emptyset)$-reduced.   If $\bbw{w}$ were  not $(X,\{i\})$-reduced,  by Proposition~\ref{prop:dbcoset}(iii)  it would be not $(\emptyset,\{i\})$-reduced,  and  we could write $\bbw{w} = (\bbw{w}\bbw{s}_i)\bbw{s}_i$ with $\ell_I(\bbw{w}) = \ell_I(\bbw{w}\bbw{s}_i) + 1$,  which contradicts the fact that $\bbw{w}$ is $(X,\emptyset)$-reduced.  The last implication then follows from  Proposition~\ref{prop:fourretout2}(iii) : let  $w$ be a reduced representative word of $\bbw{w}$.  if  $\bbw{w}$ is  $(X,\{i\})$-reduced and $\bbw{w}\,\bbw{s}_i$ is not  $(X,\emptyset)$-reduced,  then there is $s_j$ in $S_X$ so that $s_jw$ and $ws_i$ are reduced words,  but $s_jws_i$ is not.  Then $\bbw{w}\;\bbw{s_i} = \bbw{s_j} \bbw{w}$.
\end{proof}

 The following result is implicit in~\cite[\S1 Sec.~4]{Bou1968} and is well-known from specialists.  As we will need it when proving Proposition~\ref{main:prop2},  we prove it for completness. 
\begin{lm} \label{lem:base2}  Let $i,j$ be in $I$ distinct with $\{i,j\}$ in $E$.  Set $m = m(i,j)$.  Let $\underline{s}_{i,j} =  \underbrace{s_is_js_i\cdots}_{m \textrm{ terms}}$ and $\underline{s}_{j,i} =  \underbrace{s_js_is_j\cdots}_{m  \textrm{ terms}}$.   Then,  for all $n$ with $1\leq n \leq m$ we have $\bbw{r_n}(\underline{s}_{i,j}) = \bbw{r_{m-n+1}}(\underline{s}_{j,i})$.
\end{lm} 
 \begin{proof} By formula~(15) in~\cite[\S1 Sec.~4]{Bou1968} we have $\bbw{r_n}(\underline{s}_{i,j}) = (\bbw{s_is_j})^{n-1}\bbw{s_i}$ and $\bbw{r_{m-n+1}}(\underline{s}_{j,i}) = (\bbw{s_js_i})^{m-n}\bbw{s_j}$.  So $\bbw{r_n}(\underline{s}_{i,j}) (\bbw{r_{m-n+1}}(\underline{s}_{j,i}))^{-1} = (\bbw{s_is_j})^{n-1}\bbw{s_is_j}(\bbw{s_is_j})^{m-n} = (\bbw{s_is_j})^{m} = 1 $. 
 \end{proof}
\section{The retraction map}
\begin{df} Let $w = s_{i_1}\cdots s_{i_k}$ be a word on $S$.   Consider the notation in Definition~\ref{def:reflecset}.  Let $X\subseteq I$.  
\begin{enumerate}
\item  We set  $r_{j,X}(w) =  r_j(w)$ when $\bbw{r_j}(w)$ lies in $W_X$,  and  $r_{j,X}(w) = \varepsilon$,  the empty word,  otherwise.    When no confusion is possible, we will write $r_{j,X}$  for $r_{j,X}(w)$.  By $R_X(w)$,  we denote the $k$-uple of words on $S$ defined by $R_X(w) = (r_{1,X}, r_{2,X},\ldots, r_{k,X})$.  By $\bbw{R_X}(w) = (\bbw{r_{1,X}}(w), \bbw{r_{2,X}}(w),\ldots, \bbw{r_{k,X}}(w))$ we denote the corresponding $k$-uple of elements of $W_X$.  
\item  By $\widehat{R}_X(w)$  we denote the subsequence of $R_X(w)$ obtained by keeping the  nonempty words only and by $\bbw{\widehat{R}_X}(w)$  we denote the corresponding subsequence of $\bbw{R_X}(w)$. 
\end{enumerate}
\end{df} 
 Note that for $w$ in $X^*$,  one has $ \bbw{\widehat{R}_X}(w) = \bbw{R_X}(w) = \bbw{R}(w)$.   Let  $w = s_{i_1}\cdots s_{i_k}$ be  a word on $S$ and $1 < j_1<\cdots <j_n \leq k$.  In the following proposition,  by $s_{i_1}\cdots\hat{s}_{i_{j_1}} \cdots\hat{s}_{i_{j_2}} \cdots \hat{s}_{i_{j_{n}}} \cdots s_{i_k}$ we denote the word obtained from $w$ by removing the letters with a hat.  For instance in Example~\ref{exe:cox1}  $s_as_b\hat{s}_cs_as_c\hat{s}_b = s_as_bs_as_c$. 
 
 The first crucial result is the following:
\begin{prp} \label{prop:crucial}  Let $X\subseteq I$ and $w= s_{i_1}\cdots s_{i_k}$ be a non-empty word on $S$.   Write $\widehat{R}_X(w) =  (r_{j_1,X}, r_{j_2,X},\ldots, r_{j_p,X})$.   Set $\bbw{w_0} = 1$ and for $1\leq j \leq k$ write $\bbw{s_{i_1}\cdots s_{i_j}} = \bbw{v_j}\bbw{w_j}$ with $\bbw{v_j}$ in $W_X$ and $\bbw{w_j}$  that is $(X,\emptyset)$-reduced.   For $1\leq n\leq p$,  let  $\bbw{t_n}  = \bbw{w_{j_n-1}}\bbw{s_{j_n}}\bbw{w_{j_n-1}}^{-1}$.
\begin{enumerate}
\item  
\begin{enumerate}
\item For $1\leq n\leq p$,  $\bbw{t_n}$ lies in $\bbw{S_X}$. 
\item Set $j_0 = 0$ and $j_{p+1} = k+1$.  For $n$ in  $\{1,\ldots, p+1\}$ and $j$ in $\{j_{n-1},\cdots,  j_n-1\}$,  we have  $\bbw{w_j}  = \bbw{s_{i_1}}\cdots\hat{\bbw{s}}_{\bbw{i_{j_1}}} \cdots\hat{\bbw{s}}_{\bbw{i_{j_2}}} \cdots \hat{\bbw{s}}_{\bbw{i_{j_{n-1}}}} \cdots \bbw{s_{i_j}}$  and $\bbw{v_j} = \bbw{t_1}\cdots \bbw{t_{n-1}}$.
\end{enumerate}
\item Set  $w_X = t_1\cdots t_p$.   
Then,  $\bbw{R}(w_X) =\bbw{\widehat{R}_X}(w)$.  
\end{enumerate} 
\end{prp}

\begin{proof} (i) (a) and (b)  are consequences of  Lemma~\ref{lem:base1}: (a) is clear.  Now,  for  any $j$ in $\{1,\ldots, k\}$ but  not in $\{j_1,\ldots, j_p\}$,   the element $\bbw{w_j}$ is $(X,\emptyset)$-reduced and $\bbw{w_{j-1}s_{j}w_{j-1}}^{-1}$ is not in $W_X$,  because  $\bbw{v_{j-1}w_{j-1}s_{j}w_{j-1}^{-1}v_{j-1}^{-1}}$ is not.  So, $\bbw{w_{j-1}s_{j}}$ is $(X,\emptyset)$-reduced.   In particular,   $\bbw{v_j} = \bbw{v_{j-1}}$ and $\bbw{w_j} = \bbw{w_{j-1}s_{j}}$.  If $j  = j_n$ for $n$ in $\{1,\ldots, p\}$,   then $\bbw{v_{j_n}w_{j_n}} =  \bbw{v_{j_n-1}w_{j_n-1}s_{i_{j_n}}}  = \bbw{v_{j_n-1}t_n w_{j_n-1}}$.  So, $\bbw{v_{j_n}} =\bbw{v_{j_n-1}t_n}$  and $\bbw{w_{j_n}} = \bbw{w_{j_n-1}}$.  Point (ii)  follows directly from (i):  write $\bbw{u_{n}} = \bbw{s_{i_{j_{n-1}+1}}\cdots s_{i_{j_n-1}}}$.  By (i)(b),  for all $n$ we have $\bbw{w_{j_n-1}} = \bbw{u_1\cdots u_{n}}$ and $\bbw{t_1t_2\cdots t_n} =  
(\bbw{u_{1}s_{j_1}u_1^{-1}}) (\bbw{u_{1}u_{2}s_{j_2}u_{2}^{-1}u_{1}^{-1}})\cdots (\bbw{u_{1}\cdots u_{n}s_{j_n}u_{n}^{-1} \cdots u_1^{-1}})  = \bbw{u_{1}s_{j_1}u_{2}s_{j_2}\cdots s_{j_{n-1}}u_{n}s_{j_n}u_{n}^{-1} \cdots u_1^{-1}} = \bbw{s_{1}\cdots s_{j_n}u_{n}^{-1} \cdots u_1^{-1}}$.   By symmetry,  we have $\bbw{t_{n-1}\cdots t_2t_1} =  \bbw{u_{1} \cdots u_{n-1}s_{j_{n-1}}\cdots s_{1} }$.  Therefore,  we deduce that $\bbw{r_n}(w_X)  = \bbw{s_{1}\cdots s_{j_n}u_{n}^{-1} \cdots u_1^{-1}u_{1} \cdots u_{n-1}s_{j_{n-1}}\cdots s_{1} }$ $= \bbw{s_{1}\cdots s_{j_n}u_{n}^{-1}s_{j_{n-1}}\cdots s_{1} }= \bbw{r_{j_n,X}}(w)$.  \end{proof}

\begin{rmq} \label{rem:rem2}
In Point~(ii),  $t_n$ is uniquely defined by  $\bbw{t_n}$.  Moreover,  $w_X$ is uniquely defined by  $\bbw{\widehat{R}_X}(w)$ since $\bbw{t_1} = \bbw{r_{j_1,X}}$;  $\bbw{t_2} = \bbw{r_{j_1,X}} \bbw{r_{j_2,X}}\bbw{r_{j_1,X}}$ ; $\ldots$  ; $\bbw{t_n} = \bbw{r_{j_1,X}}\cdots \bbw{r_{j_{n-1},X}}\bbw{r_{j_n,X}} \bbw{r_{j_{n-1},X}}\cdots \bbw{r_{j_{1},X}}$.
\end{rmq} 
 
\begin{lm} \label{lem:simple} Let $X\subseteq I$ and $w= s_{i_1}\cdots s_{i_{k}}$,  $w'= s_{i'_1}\cdots s_{i'_{k'}}$ be two non-empty words on $S$.
\begin{enumerate}
\item Let $1\leq n\leq \min(k,k')$.  If  for all  $1\leq m \leq n$, $i_m = i'_m$  then   $r_{n,X}(w) = r_{n,X}(w')$.
\item Let $1\leq n\leq k$ and $1\leq n'\leq k'$.  Assume $\bbw{s_{i_1}\cdots s_{i_{n -1 }}} = \bbw{s_{i'_1}\cdots s_{i'_{n' - 1}}}$ and $i_{n} = i'_{n'}$.  Then $\bbw{r_{n,X}}(w) = \bbw{r_{n',X}}(w')$.     
\end{enumerate}
\end{lm}
\begin{proof} (i) $r_{n,X}(w)$ and $r_{n,X}(w')$ only depend on the prefixes $s_{i_1}\cdots s_{i_{n}}$ and $s_{i'_1}\cdots s_{i'_{n}}$,  respectively. (ii)  $\bbw{r_{n,X}}(w)$ and $\bbw{r_{n',X}}(w')$ only depend on the pair $(\bbw{s_{i_1}\cdots s_{i_{n -1 }}},  s_{i_{n}})$ and the pair $(\bbw{s_{i'_1}\cdots s_{i'_{n' -1 }}},  s_{i'_{n'}})$,  respectively.  
\end{proof}
 \begin{df} \label{def:fonct} Let $X\subseteq I$.  We define the map  $\hpiix: (\Sigma_I \cup \Sigma_I^{-1})^* \to (\Sigma_X \cup \Sigma_X^{-1})^*$ in the following way.   Let $\omega = \sigma^{\varepsilon_1}_{i_1}\cdots \sigma^{\varepsilon_k}_{i_k}$ be a word on $\Sigma_I \cup \Sigma_I^{-1}$, where $\varepsilon_i$ lies in $\{\pm 1\}$.   Set $w = \theta^*_I(\omega)$ and let $w_X = s_{q_1}\cdots s_{q_p}$ be the word on $S_X$ provided by Proposition~\ref{prop:crucial} so that  $\bbw{R}(w_X) = \bbw{\widehat{R}_X}(w) =  (r_{j_1,X}(w),\ldots, r_{j_p,X}(w))$.  We set $\hpiix(\omega) = \sigma^{\varepsilon_{j_1}}_{q_1}\cdots \sigma^{\varepsilon_{j_p}}_{q_p}$.
 \end{df}
 Note that  $(\varepsilon_1\ldots,\varepsilon_k)$ and  $\bbw{\widehat{R}_X}(w)$ do not fix  $\hpiix(\omega)$ because they do not fix the exponents.  However,  $\hpiix(\omega)$ is fixed by $(\varepsilon_1\ldots,\varepsilon_k)$ and  $\bbw{R_X}(w)$.   So the sequence of  steps that define $\hpiix$ can be summarized as it follows :$$\omega \to w \to R(w) \to \bbw{R}_X(w)  \to \bbw{\widehat{R}_X}(w) =  \bbw{R}(w_X) \to w_X \to  \hpiix(\omega)$$

 \begin{prp} \label{prop:pi} Let $X\subseteq I$ and $\omega$,  $\omega'$  be two words on $\Sigma_I \cup \Sigma_I^{-1}$.   If $\omega \equiv_I\omega'$,  then $\hpiix(\omega) \equiv_X \hpiix(\omega')$. 
 \end{prp}
\begin{proof} Let $\omega = \sigma^{\varepsilon_1}_{i_1}\cdots \sigma^{\varepsilon_k}_{i_k}$ and $\omega'$  be two words on $\Sigma_I \cup \Sigma_I^{-1}$ so that $\omega \equiv_I\omega'$.  Set $w = \theta_I^*(\omega) = s_{i_1}\cdots s_{i_k}$ and   $w' = \theta_I^*(\omega')$.   Consider the notation of Definition~\ref{def:fonct}.  In particular,  $\widehat{R}_X(w) =  (r_{j_1,X}(w),\ldots, r_{j_p,X}(w))$; $w_X = s_{q_1}\cdots s_{q_p}$  and  $\hpiix(\omega) = \sigma^{\varepsilon_{j_1}}_{q_1}\cdots \sigma^{\varepsilon_{j_p}}_{q_p}$.   Since $\omega \equiv_I\omega'$,  there exists a finite sequence $\omega_0 = \omega,  \omega_1,\ldots, \omega_n = \omega'$ of words on  $\Sigma_I \cup \Sigma_I^{-1}$  so that  $\omega_{i+1}$ is obtained from $\omega_i$ by using either a braid relation or  one of the two  relations $\sigma_i\sigma_i^{-1} = \varepsilon$ and $\sigma_i^{-1}\sigma_i =\varepsilon$.   Clearly it is enough to consider the case  $n = 1$.   Case~1: $\omega$ is transformed into $\omega'$ by using a relation $\sigma_i\sigma_i^{-1} = \varepsilon$  or $\sigma_i^{-1}\sigma_i =\varepsilon$.  Up to exchanging $\omega$ and $\omega'$ we may assume that $\omega'  =  \sigma^{\varepsilon_1}_{i_1}\cdots  \sigma^{\varepsilon_{n-1}}_{i_{n-1}}\sigma_i^{\rho}\sigma_i^{-\rho}\sigma_{i_{n}}^{\epsilon_{n}}  \cdots \sigma^{\varepsilon_k}_{i_k}$ with $\rho\in \{-1,1\}$.   Then, $w' = s_{i_1}\cdots s_{i_{n-1}}s_is_is_{i_{n}}\cdots s_{i_k}$.  By Lemma~\ref{lem:simple}(i),  for $m\leq n-1$ we have $\bbw{r_{m,X}}(w') = \bbw{r_{m,X}}(w)$ and,  by Lemma~\ref{lem:simple}(ii),  for  $m \geq  n$,  we have  $\bbw{r_{m,X}}(w) = \bbw{r_{m+2,X}}(w')$.  Moreover,  $\bbw{r_{n}}(w') = \bbw{r_{n+1}}(w')$.   So either both  $\bbw{r_{n}}(w')$ and $\bbw{r_{n+1}}(w')$ are  in $W_X$ or both of them are not.  In the second  case,  $\bbw{R_X}(w) = \bbw{R_X}(w')$ and  $\hpiix(\omega) = \hpiix(\omega')$.  Assume the first case.  Let $w'_X \in S^*_X$ so that $\bbw{R}(w'_X) =\bbw{\widehat{R}_X}(w')$.  Since $\bbw{r_{n,X}}(w') = \bbw{r_{n+1,X}}(w')\neq 1$   there is  $0\leq m \leq p$  so that  $\bbw{\widehat{R}_X}(w') =   (\bbw{r_{j_1,X}}(w),\ldots,  \bbw{r_{j_m,X}}(w),\bbw{r_{n,X}}(w'),\bbw{r_{n,X}}(w'),\bbw{r_{j_{m+1},X}}(w),\ldots, \bbw{r_{j_p,X}}(w))$.   In particular,  there is $j$ in $X$ so that  $w'_X = s_{q_1}\cdots s_{q_m}s_js_js_{q_{m+1}}\cdots s_{q_p}$ (see Remark~\ref{rem:rem2}) and $\hpiix(\omega') = \sigma^{\varepsilon_{j_1}}_{q_1}\cdots\sigma^{\varepsilon_{j_m}}_{q_m}\sigma_j^{\rho}\sigma_j^{-\rho}\sigma^{\varepsilon_{j_{m+1}}}_{q_{m+1}} \cdots\sigma^{\varepsilon_{j_p}}_{q_p}$.  Thus,  $\hpiix(\omega) \equiv_X\hpiix(\omega')$. 
 Case~2: $\omega'$ is obtained from $\omega$ by  using a  braid relation.   For $i,j$ in $I$,  write $\underline{\sigma}_{i,j} = \underbrace{\ssi\ssj\ssi\cdots}_{m(i,j) \textrm{ terms}}$ and  $\underline{s}_{i,j} = \underbrace{s_is_js_i\cdots}_{m(i,j) \textrm{ terms}}$.   We can write $\omega = \sigma^{\varepsilon_1}_{i_1}\cdots  \sigma^{\varepsilon_{n-1}}_{i_{n-1}}\underline{\sigma}_{i,j}   \sigma^{\varepsilon_{n+m}}_{i_{n+m}} \cdots \sigma^{\varepsilon_k}_{i_k}$ and  $\omega' = \sigma^{\varepsilon_1}_{i_1}\cdots  \sigma^{\varepsilon_{n-1}}_{i_{n-1}}\underline{\sigma}_{j,i}  \sigma^{\varepsilon_{n+m}}_{i_{n+m}} \cdots \sigma^{\varepsilon_k}_{i_k}$ with $i = i_n$, $j = i_{n+1}$ and $m = m(i,j)$.  It follows that we have  $w = s_{i_1}\cdots  s_{i_{n-1}}\underline{s}_{i,j} s_{i_{n+m}} \cdots s_{i_k}$ and $w' = s_{i_1}\cdots  s_{i_{n-1}}\underline{s}_{j,i} s_{i_{n+m}} \cdots s_{i_k}$.    As in Case~1,  by Lemma~\ref{lem:simple},   $\bbw{r_{h}}(w) = \bbw{r_{h}}(w')$ when  either $1\leq h \leq n-1$ or $n+m\leq h \leq k$.  Since $\underline{\sigma}_{i,j}$ is a reduced word,  all the elements of $\bbw{R}(\underline{\sigma}_{i,j} )$ are distinct (see Proposition~\ref{prop:fourretout1}).  It follows that  $\bbw{r_{n}}(w),\bbw{r_{n+1}}(w), \ldots, \bbw{r_{n+m-1}}(w)$ are all distinct.  Moreover,  by Lemma~\ref{lem:base2},   for $h$ in $\{1,\ldots,m\}$ we have $\bbw{r_{h}}(\underline{\sigma}_{i,j} ) = \bbw{r_{m-h+1}}(\underline{\sigma}_{j,i})$, and $\bbw{r_{n-1+h}}(w) = \bbw{r_{n+m-h}}(w')$.   Therefore,  if none of the elements $\bbw{r_{n}}(w),\bbw{r_{n+1}}(w), \ldots, \bbw{r_{n+m-1}}(w)$ lie in $W_X$ then  $\bbw{R_X}(w) = \bbw{R_X}(w')$ and  $\hpiix(\omega) \equiv_X \hpiix(\omega')$.  If only one of them lies in $W_X$,  say $\bbw{r_{n-1+\tilde{h}}}(w)$,  then   $\bbw{R_X}(w)$ and  $\bbw{R_X}(w')$ are not equal,  but differ only at positions $n-1+\tilde{h}$ and $n-1+m-\tilde{h}$ and $\bbw{\widehat{R}_X}(w')  = \bbw{\widehat{R}_X}(w)$.  Then $\hpiix(\omega)$  and $\hpiix(\omega')$ because the exponents in position $n-1+\tilde{h}$ and $n-1+m-\tilde{h}$,  respectively,  are  both equal to~$1$.  Assume finally that at liest two  terms among $\bbw{r_{n}}(w),\bbw{r_{n+1}}(w), \ldots, \bbw{r_{n+m-1}}(w)$ lie in $W_X$.   By construction,  all the terms of the previous list belong to $(\bbw{s_{i_1}\cdots s_{i_{n-1}}})W_{i,j}(\bbw{s_{i_1}\cdots s_{i_{n-1}}})^{-1}$.   Since  $(\bbw{s_{i_1}\cdots s_{i_{n-1}}})W_{i,j}(\bbw{s_{i_1}\cdots s_{i_{n-1}}})^{-1} \cap W_X$ contains at least two distinct elements different from~$1$,   Corollary~\ref{cor:base2} implies that $(\bbw{s_{i_1}\cdots s_{i_{n-1}}})W_{i,j}(\bbw{s_{i_1}\cdots s_{i_{n-1}}})^{-1}$ is included in the parabolic subgroup $W_X$,  and  $\bbw{r_{n}}(w),\bbw{r_{n+1}}(w), \ldots, \bbw{r_{n+m-1}}(w)$ all belong to $W_X$.   Recall the notations given in the introduction for  $\widehat{R}_X(w)$, $w_X$  and  $\hpiix(\omega)$ .  Let $p'$ be such that $j_{p'} = i_n$.  Set $i' = q_{p'}$ and $j' =  q_{p'+1}$.  If $\bbw{v_j}$ and $\bbw{w_j}$ are defined like in Proposition~\ref{prop:crucial},  then,  by Point(i)(b) of this proposition,  $\bbw{w_n} = \bbw{w_{n+1}} = \cdots = \bbw{w_{n+m-1}}$.  Now by Remark~\ref{rem:base2} we have  $m(i',j') = m(i_n,i_{n+1})$.    It follows that $s_{q_{p'}}\cdots s_{q_{p'+m-1}} = \underline{s}_{i',j'}$,  $w_X = s_{q_1}\cdots s_{q_{p'-1}}\underline{s}_{i',j'}s_{q_{p'+m}}\cdots s_{q_p}$  and  $w'_X = s_{q_1}\cdots s_{q_{p'-1}}\underline{s}_{j',i'}s_{q_{p'+m}}\cdots s_{q_p}$.   Thus,  $\hpiix(\omega) = \sigma_{q_1}^{\varepsilon_{j_1}}\cdots \sigma_{q_{p'-1}}^{\varepsilon_{j_{q'-1}}} \underline{\sigma}_{i',j'}\sigma_{j_{p'+m}}^{\varepsilon_{j_{p'+m}}}\cdots \sigma^{\varepsilon_{j_p}}_{p}$   and $\hpiix(\omega') = \sigma_{q_1}^{\varepsilon_{j_1}}\cdots \sigma_{q_{p'-1}}^{\varepsilon_{j_{q'-1}}} \underline{\sigma}_{j',i'}\sigma_{j_{p'+m}}^{\varepsilon_{j_{p'+m}}}\cdots \sigma^{\varepsilon_{j_p}}_{p}$   and we are done.   
\end{proof}
We are now ready to prove Propositions~\ref{main:prop2} and~\ref{main:prop1}.
\begin{proof}[Proof of Proposition~\ref{main:prop2}]  Let $\omega$ be a word on $\Sigma_I \cup \Sigma_I^{-1}$.   Set $w  = \theta_I^*(\omega)$. The length of $\omega$ is equal  to the length of $w$,  which in turn is equal  to the number of terms of $\bbw{R}(w)$.  Similarly  the length of $\hpiix(\omega)$ is equal to the length of $\theta_X^*(\hpiix(\omega))$.  But,  by definition,  the latter word is $w_X$   whom length is equal to the number of terms of~$\bbw{R}(w_X)$, that is of~$\bbw{\widehat{R}_X}(w)$.   Thus,  $\ell_X(\hpiix(\omega)) \leq \ell_I(\omega)$. The  equality holds when $\bbw{R}(w_X) = \bbw{R}(w)$,  which is  equivalent to have $w_X = w$,   which, in turn,  is equivalent to have $ \hpiix(\omega) = \omega$.  So Point~(i) holds.   Let $\omega'$ be another word on $\Sigma_I \cup \Sigma_I^{-1}$ and  set $w'  = \theta_I^*(\omega')$.   By definition,  $\bbw{R_X}(ww')$ starts with $\bbw{R_X}(w)$.   Therefore,   $w_X$ is a prefix of $(ww')_X$ and  Point~(ii)  follows.  Point (iii)  follows from Proposition~\ref{prop:pi}.   Since~$\theta_X(\piix(\bbw{\omega})) = \bbw{w_X}$,  Point (iv) follows from Proposition~\ref{prop:crucial} (i)(b) and (ii).   Assume~$\bbw{\omega}$ lies in $CA_I$. Then,  $\bbw{w} = 1$ in $W_I$.  By Proposition~\ref{prop:fourretout1},  $\bbw{N}(w)$ has to be empty and all terms of $\bbw{R}(w)$ appear an even number of times.  But in this case,  all of  its terms that are in $W(X)$ appear an even number of times, too,  and  $\bbw{N}(w_X)$ is empty.   Then $\bbw{w_X} =1$.  But  $\theta_X^*(\hpiix(w)) = w_X$.  Then, $\piix(\bbw{\omega})$ lies in $CA_X$ and the map~$\piix$ restricts to an application from $CA_I$ to $CA_X$.   The fact that $\piix : CA_I \to CA_X$ is an homomorphism,  and Point (v),  will follow from  Point~(viii).   Point (vi) is clear from the definition: with the notation of Definition~\ref{def:fonct},  if all the $\epsilon_n$  are equal to $1$,  so are the $\varepsilon_{j_n}$.   Let $w'$ be as above,  set $k = \ell_S(w)$  and write  $(ww')_X = w_Xw''$.   If $\bbw{\omega}$ lies in $CA_I$ or if $\omega$  is on  $\Sigma_X \cup \Sigma_X^{-1}$,  then for $n\geq k+1$,  $\bbw{r_n}(ww')$ lies in $W_X$ if and only if $ \bbw{w}\, \bbw{r_{n-k}}(w')\bbw{w}^{-1}$ lies in $W_X$.  So in both cases,  $\bbw{R_X}(ww')$ is the concatenation of $\bbw{R_X}(w)$ and  $\bbw{w}\bbw{R_X}(w')\bbw{w}^{-1}$.   It follows that $\bbw{R}((ww')_X)$ is the concatenation of $\bbw{R}(w_X)$ and $\bbw{w}\bbw{R}(w'_X)\bbw{w}^{-1}$.  By Remark~\ref{rem:conca},  we  deduce that $\bbw{R}(w'') =( \bbw{w_X}^{-1}\bbw{w})\bbw{R}(w'_X)(\bbw{w}^{-1}\bbw{w_X})$.  But if $\bbw{\omega}$ lies in $CA_I$,  then $\piix(\bbw{w})$ lies in $CA_X$ and $\bbw{w}  = \bbw{w_X} = 1$.  On  the other hand,  if $\omega$  is a word on  $\Sigma_X \cup \Sigma_X^{-1}$,  then $w_X = w$  and $\bbw{w_X}^{-1}\bbw{w} = 1$.  So in both cases,   $\bbw{R}(w'')  = \bbw{R}(w_X')$,  which implies $w'' = w_X'$.  This proves Points (vii) and (viii).  Note that the exponents in  $\hpiix(\omega\omega')$ are the expected ones because $\bbw{R_X}(ww')$ is the concatenation of $\bbw{R_X}(w)$ and  $\bbw{w}\bbw{R_X}(w')\bbw{w}^{-1}$.  Assume $Y\subseteq I$ and $\omega$ is a word on $\Sigma_Y \cup \Sigma_Y^{-1}$.  Then $w$ is a word on $S_Y$ and the terms  of $\bbw{R}(w)$ are in $W_Y$.  Therefore,  they are  in $W_X$ if and only if they are in $W_X\cap W_Y$, that is in $W_{X\cap Y}$.   So the terms of  $\bbw{R_X}(w)$ are in $W_{X\cap Y}$ and  $w_X$ is a word on $S_{X\cap Y}$.  Point (ix) follows.   Finally,  assume $Y\subseteq I$ and $\omega = \sigma^{\varepsilon_1}_{i_1}\cdots \sigma^{\varepsilon_k}_{i_k}$  is a word on $\Sigma_I \cup \Sigma_I^{-1}$.  Set $w = \theta_I^*(\omega) = s_{i_1}\cdots s_{i_k}$ and,  as in Definition~\ref{def:fonct},  $\widehat{R}_X(w) =  (r_{j_1,X}(w),\ldots, r_{j_p,X}(w))$; $w_X = s_{q_1}\cdots s_{q_p}$  and  $\hpiix(\omega) = \sigma^{\varepsilon_{j_1}}_{q_1}\cdots \sigma^{\varepsilon_{j_p}}_{q_p}$.   Let us prove (x).   By symmetry,  it is enough to prove that   $\pi^*_{I,Y}(\pi^*_{I,X}(\omega)) = \pi^*_{I,X\cap Y}(\omega)$.  Write 
$\pi^*_{I, Y}(\pi^*_{I,X}(\omega)) = \sigma^{\varepsilon_{j'_1}}_{q'_1}\cdots \sigma^{\varepsilon_{j'_{p'}}}_{q'_{p'}}$ with $j'_1,\ldots, j'_{p'}$ in $\{j_1,\ldots,  j_p\}$.   Write  $\pi^*_{I, Y\cap X}(\omega) = \sigma^{\varepsilon_{j''_1}}_{q''_1}\cdots \sigma^{\varepsilon_{j''_{p''}}}_{q''_{p''}}$ with $j''_1,\ldots, j''_{p''}$ in $\{1,\ldots,  k\}$.    We have  $(w_{X})_Y = s_{q'_1}\cdots s_{q'_{p'}}$  and   $w_{X\cap Y} = s_{q''_1}\cdots s_{q''_{p''}}$.   By construction $\bbw{\widehat{R}_Y}(\theta_I^*(\hpiix(\omega)) )$  is obtained from  $\bbw{R}(\theta_I^*(\hpiix(\omega)))$  by removing the terms that do not belong to $W_Y$.  But $\theta_I^*(\hpiix(\omega)) = w_X$ and  $\bbw{R}(w_X)$ is obtained from  $\bbw{R}(w)$ by removing the terms that do not belong to $W_X$.  So,  $\bbw{\widehat{R}_Y}(\theta_I^*(\hpiix(\omega)))$ is obtained from $\bbw{R}(w)$ by removing the terms that do not belong to $W_X\cap W_Y$,  that is to $W_{X\cap Y}$.  Thus $\bbw{{R}}((w_{X})_Y) =  \bbw{\widehat{R}_Y}(\theta_I^*(\hpiix(\omega))) = \bbw{\widehat{R}_{Y\cap X}}(w) = \bbw{{R}}(w_{Y\cap X})$ and  $(w_{X})_Y =  w_{Y\cap X}$.  In particular $p' = p''$ and $q'_{\tilde{p}} = q''_{\tilde{p}}$ for $\tilde{p}$ in $1,\cdots,  p'$.  Now,  let $\tilde{p}$ be in $1,\ldots, p'$.  Then,  $j'_{\tilde{p}}$ is the position  of the ${\tilde{p}}^{th}$ term  of $\bbw{R}(w)$ that belongs to $W_Y\cap W_X$,  that is to $W_{X\cap Y}$.  Therefore,  $j'_{\tilde{p}} = j''_{\tilde{p}}$ and  $\pi^*_{I, Y}(\pi^*_{I,X}(\omega))  = \pi^*_{I,X\cap Y}(\omega)$.  Hence,  (x) holds.  Finally,  by Point~(ix),  Point~(xi) is a special case of Point~(x).    Indeed,   since  $\pi^*_{I,Y}(\omega)$ is a word on $\Sigma_Y \cup \Sigma_Y^{-1}$,   we have  $\pi^*_{I, X}(\pi^*_{I,Y}(\omega)) =  \pi^*_{Y, X}(\pi^*_{I,Y}(\omega))  $.      
\end{proof}
\begin{proof}[Proof of Proposition~\ref{main:prop1}] 
 Let $\omega,\omega'$ be in $(\Sigma_X\cup \Sigma_X^{-1})^*$ and such that $ \omega\equiv_I  \omega'$.  By Proposition~\ref{main:prop2}(i),   $\hpiix(\omega) = \omega$ and  $\hpiix(\omega') = \omega'$ and by Proposition~\ref{prop:pi},  $\hpiix(\omega) \equiv_X \hpiix(\omega')$.  So Point  (i) holds.   Let $Y\subseteq I$.  Clearly $A_{X\cap Y}\subseteq A_X\cap A_Y$.  Let $\bbw{w}$ be in $A_X\cap A_Y$.  Since  $\bbw{w}$ lies in $A_Y$,  by  Proposition~\ref{main:prop2}(ix), $\piix(\bbw{w})$ lies in $A_{X\cap Y}$.   Since $\bbw{w}$ lies in $A_X$,   by Corollary~\ref{cor:main} (i) (that we can apply as Point (i) is proved),   $\piix(\bbw{\omega}) = \bbw{\omega}$.  So,  $A_X\cap A_Y\subseteq A_{X\cap Y}$ and Point~(ii) holds.   Similarly $A_X^+ \subseteq A^+\cap A_X$.   The argument to prove  the inverse inclusion is similar as the one used to prove (ii),  replacing   Proposition~\ref{main:prop2}(ix) with  Proposition~\ref{main:prop2}(vi).  Finally,  Point (iv) follows from almost the same argument : if $\omega$ is a word representative of an element of $A_X$ then,  by Proposition~\ref{main:prop2}(i) and (iii),  $\hpiix(w)$ is another  word representative of the same element whom length is not greater than the one of $\omega$.  If moreover $\omega$ is a minimal word representative,  then  $\omega$ and $\hpiix(w)$ must have the same length.  By Proposition~\ref{main:prop2}(i) this imposes $\omega = \hpiix(\omega)$.  Hence $A_X$ is convex.  
\end{proof}

\begin{prp}\label{prp:lien:paris} the map $\hpiix$ coincides with the map $\hat{\pi}_X$ defined in \cite{BlPa2022}. 
\end{prp}
\begin{proof}  Let $X\subseteq I$ and $\omega = \sigma^{\varepsilon_1}_{i_1}\cdots \sigma^{\varepsilon_k}_{i_k}$ be a word on $\Sigma_I\cup \Sigma_I^{-1}$  where $\epsilon_i$  lies in $\{\pm 1\}$.  Let  us recall the definition of $\hat{\pi}_X$ given in~\cite{BlPa2022}.   Set $w = \theta_I^*(\omega) = s_{i_1}\cdots s_{i_k}$.   For $1\leq j\leq k$,  let $\bbw{s_{i_1}\cdots s_{i_j}} = \bbw{v_j w_j}$ with $\bbw{v_j}$ in $W_X$ and $\bbw{w_i}$ that is $(X,\emptyset)$-reduced.  If $\varepsilon_j = 1$,  set $\bbw{t_j} = \bbw{w_{j-1}s_{i_j}w_{j-1}}^{-1}$.  If $\varepsilon_j = -1$,  set $\bbw{t_j} = \bbw{w_{j}s_{i_j}w_{j}}^{-1}$.  If $\bbw{t_j}$ lies in $S_X$ set $\tau_j = \sigma^{\epsilon_{j}}_{q_j}$ so that $\bbw{t_j} = \bbw{s_{q_j}}$.  Otherwiese,  set $\tau_j = \varepsilon$,  the empty word.     Then $\hat{\pi}_X(\omega)$ is defined by  $\hat{\pi}_X(\omega) = \tau_1\cdots \tau_k$.  Now,  if $\bbw{w_{j-1}s_{i_j}w_{j-1}^{-1}}$ lies in $S_X$ then $\bbw{w_j}= \bbw{w_{j-1}}$; otherwise $\bbw{w_j}= \bbw{w_{j-1} s_{i_j}}$ (see the proof of Proposition~\ref{prop:crucial}(i)(b)).   In both cases,   $\bbw{w_{j}s_{i_j}w_{j}^{-1}} = \bbw{w_{j-1}s_{i_j}w_{j-1}^{-1}}$.  So,  in the above definition of $\hat{\pi}_X$,  this is not necessary to distinguish whether $\varepsilon_i$ is equal to $1$ or to $-1$ and  
$\hat{\pi}_X(\omega) = \hpiix(\omega)$.  \end{proof}

\begin{cor} \label{cor:lien:paris} The map $\piix$ coincides with the map $\pi_X$ defined in \cite{BlPa2022}. 
\end{cor}

\section{Intersection of parabolic subgroups}
We turn now to the proof of Theorems~\ref{main:thm1} and~\ref{main:thm2}.
We start with two preliminary lemmas.   We recall  that~$\KKI :  W_I \to A_I$ is the set section  of the morphism~$\Theta_I : A_I\to W_I$.   
 
\begin{lm}  \label{sec3:lem1} Let $X$ be a subset of $I$ and $\bbw{\omega},\bbw{\omega'}$ be in $A_I$ with $\omega$ in $CA_I$.   If $\bbw{\omega}\bbw{\omega'}\bbw{\omega}^{-1}$ lies in $A_X$ then  $\bbw{\omega}\bbw{\omega'}\bbw{\omega}^{-1} = \piix(\bbw{\omega}\bbw{\omega'}\bbw{\omega}^{-1}) = \piix(\bbw{\omega})\piix(\bbw{\omega'})\piix(\bbw{\omega})^{-1}$.
\end{lm}
\begin{proof} Since $\bbw{\omega}\bbw{\omega'}\bbw{\omega}^{-1}$ lies in $A_X$,  the first equality follows from  Corollary~\ref{cor:main}(i).  We have $\bbw{\omega}\bbw{\omega'} =(\bbw{\omega}\bbw{\omega'}\bbw{\omega}^{-1}) \bbw{\omega}$.  So $\piix(\bbw{\omega}\bbw{\omega'}) = \piix((\bbw{\omega}\bbw{\omega'}\bbw{\omega}^{-1}) \bbw{\omega})$.  By Proposition~\ref{main:prop2}(viii)  $\piix(\bbw{\omega}\bbw{\omega'}) =\piix(\bbw{\omega})\piix(\bbw{\omega'})$ and by Corollary~\ref{cor:main}(ii),  $\piix((\bbw{\omega}\bbw{\omega'}\bbw{\omega}^{-1}) \bbw{\omega}) = (\bbw{\omega}\bbw{\omega'}\bbw{\omega}^{-1}) \piix(\bbw{\omega})$.
\end{proof}
Note that Lemma~\ref{sec3:lem1} extends~\cite[Lemma~2.4]{BlPa2022}.
\begin{lm} \label{lem:lem22} Let $X,Y\subseteq I$ and $\bbw{w}\in W_I$ that is $(\emptyset,Y)$-reduced.  Write $\bbw{w}   = \bbw{w_1}\bbw{w_2}$ with $\bbw{w_1}$ in $W_X$  and $\bbw{w_2}$ $(X,Y)$-reduced.    Then,  $\piix(\KKI(\bbw{w})A_Y)  =  \piix(\KKI(\bbw{w})) A_{X_1}$ with   $S_{X_1} = S_X\cap \bbw{w_2}S_Y\bbw{w_2}^{-1}$. 
\end{lm} 
\begin{proof} Let $\bbw{\omega} = \KKI(\bbw{w})$,  $\bbw{\omega_1} = \KKI(\bbw{w_1})$  and $\bbw{\omega_2} = \KKI(\bbw{w_2})$.  Then,   $\ell_S(\bbw{w}) = \ell_S(\bbw{w_1})+\ell_S(\bbw{w_2})$ and $\bbw{\omega} = \bbw{\omega_1}\bbw{\omega_2}$ with $\bbw{\omega_1}$ in~$A_X$.    First,  we have $\piix(\bbw{\omega_2}) =1$ and  $\piix(\bbw{\omega}) = \piix(\bbw{\omega_1}) = \bbw{\omega_1} = \piix(\bbw{\omega_1}) \piix(\bbw{\omega_2}) $ (see Proposition~\ref{main:prop2}(iv)).  For any $\bbw{\omega'}$ in $A_Y$,  by Corollary~\ref{cor:main}(ii),  $\piix(\bbw{\omega}\bbw{\omega'}) = \piix(\bbw{\omega_1}) \piix(\bbw{\omega_2}\bbw{\omega'})$.  So we may assume $\bbw{w_1} = 1$ and  $\bbw{w} = \bbw{w_2}$.  Now,  let $w = s_{i_1}\cdots s_{i_k}$ be a representative word of $\bbw{w}$ of minimal length.  Then $\sigma_{i_1}\cdots \sigma_{i_k}$ is a word representative of $\bbw{w}$.   Let  $\bbw{\omega'}$ be in~$A_Y$ and $\omega' = \sigma^{\varepsilon_1}_{i'_1}\cdots \sigma^{\varepsilon_{k'}}_{i'_{k'}}$ be a representative word of  $\bbw{\omega'}$ of minimal length.  Set $w' = \hpiix(\omega')  =  s_{i'_1}\cdots s_{i'_{k'}}$.  Then all the $i'_{j}$ lie in $Y$ and $\sigma_{i_1}\cdots \sigma_{i_k}\sigma^{\varepsilon_1}_{i'_1}\cdots \sigma^{\varepsilon_{k'}}_{i'_{k'}}$ is a representative word of $\bbw{\omega}\bbw{\omega'}$.  Since $\bbw{w}$ is $(X,\emptyset)$-reduced and $w$ is a minimal word representative of $\bbw{w}$,  the $k$ first entries of $\bbw{R_X}(ww')$ are equal to $1$ and all the next ones lie in $\bbw{w}W_Y\bbw{w}^{-1} \cap W_X$.  By  Proposition~\ref{prop:base2}, there is  $X_1\subseteq X$ and $Y_1\subseteq Y$ such that $\bbw{w}S_{Y_1} = S_{X_1}\bbw{w}$ and $\bbw{w}W_Y\bbw{w}^{-1} \cap W_X = W_{X_1} = \bbw{w}W_{Y_1}\bbw{w}^{-1}$.  Then,  $(ww')_X$ is a word on $S_{X_1}$.  Since $\theta^*_I(\omega\omega') = ww'$,   we get that $\piix(\bbw{\omega}\bbw{\omega'})$ lies in  $A_{X_1}$ and  $\piix(\KKI(\bbw{w})A_Y) \subseteq A_{X_1}$.  Conversely,  let $\bbw{\omega''}$ lie in  $A_{X_1}$.   Then $\bbw{\omega}^{-1}\bbw{\omega''}\bbw{\omega}$ lies in $A_{Y_1}$.  Indeed,  since $\bbw{w}S_{Y_1} = S_{X_1}\bbw{w}$  and $\bbw{w}$ is $(X,Y)$-reduced,  for any $i$ in $X_1$ there is $j$ in $Y_1$ so that  $\bbw{w}\bbw{s_{j}} = \bbw{s_i}\bbw{w}$ with $\ell_I(\bbw{w}\bbw{s_{i}} ) = \ell_I(\bbw{s_{j}}\bbw{w}) = \ell_I(\bbw{w)} +1$.  Therefore,  $\bbw{\omega}\bbw{\sigma_{j}} = \bbw{\sigma_i}\bbw{\omega}$.  Now,   $\piix(\KKI(\bbw{w}) \bbw{\omega}^{-1}\bbw{\omega''}\bbw{\omega}) = \piix(\bbw{\omega''}\bbw{\omega}) =  \piix(\bbw{\omega''})\piix(\bbw{\omega}) =  \bbw{\omega''}$.  So,  the other inclusion holds. 
\end{proof}

The following result proves Theorem~\ref{main:thm2}. 
\begin{prp} \label{prop:lem22bis} Let $X,Y\subseteq I$ and $\bbw{w}\in W_I$.  Let  $\bbw{w} = \bbw{w_1}\bbw{w_2}\bbw{w'_2}$ with $\bbw{w_1}$ in $W_X$,  $\bbw{w'_2}$ in $W_Y$ and $\bbw{w_2}$ $(X,Y)$-reduced.    Write $\bbw{\omega} = \KKI(\bbw{w})$,    $\bbw{\omega_1} = \KKI(\bbw{w_1})$,  and $\bbw{\omega'_2} = \KKI(\bbw{w'_2})$.  Let $X_1, Y_1$ be defined as in Proposition~\ref{prop:base2}.  Then $\bbw{\omega} A_Y \bbw{\omega}^{-1}\cap A_X = \bbw{\omega_1}A_{X_1}{\bbw{\omega_1}}^{-1}$ and  $\bbw{\omega}^{-1} A_X \bbw{\omega}\,\cap A_Y = {\bbw{\omega'_2}}^{-1}A_{Y_1}\bbw{\omega'_2}$.  
\end{prp} 

\begin{proof}   We have~$\ell_I(\bbw{w})  = \ell_I(\bbw{w_1})+\ell_I(\bbw{w_2})+\ell_I(\bbw{w'_2})$ and  $\bbw{\omega}= \bbw{\omega_1}\bbw{\omega_2}\bbw{\omega'_2}$ with $\bbw{\omega_1}$ in $A_X$ and  $\bbw{\omega'_2}$ in $A_Y$.  So $\bbw{\omega} A_Y \bbw{\omega}^{-1}\cap A_X = \bbw{\omega_1}\bigg(\bbw{\omega_2} A_Y \bbw{\omega_2}^{-1}\cap A_X\bigg){\bbw{\omega_1}}^{-1}$.  Similarly,  $\bbw{\omega}^{-1} A_X \bbw{\omega}\cap A_Y = {\bbw{\omega'_2}}^{-1}\bigg(\bbw{\omega_2}^{-1} A_X \bbw{\omega_2}\cap A_Y\bigg){\bbw{\omega'_2}}$.  So,  we may assume $\bbw{w_1} = \bbw{w'_2} = 1$ and $\bbw{w}$ is $(X,Y)$-reduced.  In this case,  $\theta_I(\bbw{\omega}) = \bbw{w}$ and $\piix(\bbw{\omega}) = 1$ (see Proposition~\ref{main:prop2}(iv)).   Let $\bbw{\omega'}$ lie in $A_Y$.  Set $\bbw{\omega''} = \bbw{\omega} \bbw{\omega'} \bbw{\omega}^{-1}$ and assume    $\bbw{\omega''}$ lies in $A_{X}$.  Then 
$\bbw{\omega} \bbw{\omega'} = \bbw{\omega''} \bbw{\omega}$ and $\piix(\bbw{\omega} \bbw{\omega'}) = \piix( \bbw{\omega''} \bbw{\omega})  = \piix( \bbw{\omega''}) \piix(\bbw{\omega}) =   \bbw{\omega''}$.  But,  by Lemma~\ref{lem:lem22},   $\piix(\bbw{\omega} \bbw{\omega'})$ lies in $A_{X_1}$.   So $\bbw{\omega} A_Y \bbw{\omega}^{-1}\cap A_X \subseteq A_{X_1}$.  Conversely,   $A_{X_1}$ is included in $A_{X}$  and,   as seen in the end of the proof of Lemma~\ref{lem:lem22},  $\bbw{\omega}^{-1}A_{X_1}\bbw{\omega} = A_{Y_1}$.  So  $A_{X_1} \subseteq  \bbw{\omega}A_{Y}\bbw{\omega}^{-1}$  and  the other inclusion holds. Thus,  $\bbw{\omega} A_Y \bbw{\omega}^{-1}\cap A_X = A_{X_1}$ and $\bbw{\omega}^{-1} A_X \bbw{\omega}\,\cap A_Y = \bbw{\omega}^{-1} (A_X\cap \bbw{\omega} A_Y\bbw{\omega}^{-1})\bbw{\omega} = \bbw{\omega}^{-1} A_{X_1}\bbw{\omega} = A_{Y_1}$.   
\end{proof}

\begin{proof}[Proof of Theorem~\ref{main:thm1}]
Assume Conjecture~\ref{conj:conj2} holds. 
Let $X,Y$ be in $I$ and $\bbw{\omega}$ be in $A_I$.   We want to prove that $(\bbw{\omega}A_Y\bbw{\omega}^{-1}) \cap A_X$ is a parabolic subgroup of $A_I$.  Write $\theta_I(\bbw{\omega}) = \bbw{w} = \bbw{w_1}\bbw{w_2}\bbw{w'_2}$ with $\bbw{w_1}$ in $W_X$,  $\bbw{w'_2}$ in $W_Y$ and $\bbw{w_2}$ $(X,Y)$-reduced.   Then  $$(\bbw{\omega}A_{Y} \bbw{\omega}^{-1})\cap A_X =  \KKI(\bbw{w_1})\bigg(\big({\KKI(\bbw{w_1}})^{-1}\bbw{\omega}\,\KKI(\bbw{w'_2})^{-1}\big) A_{Y}\big( \KKI(\bbw{{w'_2}}){\bbw{\omega}}^{-1}\KKI(\bbw{w_1})\big)\cap A_X  \bigg){\KKI(\bbw{w_1})}^{-1}$$  
and $\theta_I(\big({\KKI(\bbw{w_1}})^{-1}\bbw{\omega}\,\KKI(\bbw{{w'}_2})^{-1}) = {\bbw{w_1}}^{-1}\bbw{w}\bbw{{w'}_2}^{-1} = \bbw{w_2}$.  So we can assume  $\bbw{w_1} = \bbw{w'_2} = 1$ and $\bbw{w}$ is $(X,Y)$-reduced.   We set  $\bbw{\omega_2} = \KKI(\bbw{w})$.   Then $\theta_I(\bbw{\omega_2}) = \bbw{w}$.   Let  $\bbw{\omega'}$ be in $A_Y$.   As already seen, $\bbw{w}$ being  $(X,Y)$-reduced,   we have $\piix(\bbw{\omega_2})= 1$ and,  by Lemma~\ref{lem:lem22},  $\piix(\bbw{\omega_2}A_Y) =  A_{X_1}$
where  $S_{X_1} = S_X\cap \bbw{w}S_Y\bbw{w}^{-1}$.  Then  $\piix(\bbw{\omega_2}\bbw{\omega'})$ lies in~$A_{X_1}$.
 Assume $\bbw{\omega}\bbw{\omega'}\bbw{\omega}^{-1}$ belongs to $A_X$.  Then $\piix(\bbw{\omega}\bbw{\omega'}\bbw{\omega}^{-1})  = \bbw{\omega}\bbw{\omega'}\bbw{\omega}^{-1}$.  On the other hand,  $\theta_I(\bbw{\omega}\,\bbw{\omega}_2^{-1}) = \theta_I(\bbw{\omega})\theta_I(\bbw{\omega}_2^{-1}) = \bbw{w}\,\bbw{w}^{-1}  = 1$ and  $\bbw{\omega}\,\bbw{\omega}_2^{-1}$ lies  in $CA_I$.  Now,   $(\bbw{\omega}\bbw{\omega'}\bbw{\omega}^{-1} ) \big(\bbw{\omega}\, \bbw{\omega_2}^{-1}\big) \bbw{\omega_2}    =  (\bbw{\omega}\, \bbw{\omega_2}^{-1})\bbw{\omega_2}\bbw{\omega'}$.   So,   by  Proposition~\ref{main:prop2}(vii) and (viii),  applying $\piix$  we get,  
$(\bbw{\omega}\bbw{\omega'}\bbw{\omega}^{-1} )  \piix\big(\bbw{\omega}\,\bbw{\omega_2}^{-1}\big)  \piix(\bbw{\omega_2})    =   \piix\big(\bbw{\omega}\, \bbw{\omega_2}^{-1}\big) \piix\big(\bbw{\omega_2}\bbw{\omega'}\big) \in \piix\big(\bbw{\omega}\, \bbw{\omega_2}^{-1}\big) A_{X_1}$    and $\bbw{\omega}\bbw{\omega'}\bbw{\omega}^{-1}$ lies in   
$\piix\big(\bbw{\omega}\, \bbw{\omega_2}^{-1}\big)A_{X_1}\,\big(\piix\big(\bbw{\omega}\, \bbw{\omega_2}^{-1}\big)\big)^{-1}$.  Hence,  $\bbw{\omega} A_Y \bbw{\omega}^{-1}\cap A_{X} \subseteq \piix\big(\bbw{\omega}\, \bbw{\omega_2}^{-1}\big)A_{X_1}\,\big(\piix\big(\bbw{\omega}\, \bbw{\omega_2}^{-1}\big)\big)^{-1} \subseteq A_X$ and $\bbw{\omega} A_Y \bbw{\omega}^{-1}\cap A_{X} = \bbw{\omega} A_Y \bbw{\omega}^{-1} \cap \bigg( \piix\big(\bbw{\omega}\, \bbw{\omega_2}^{-1}\big)A_{X_1}\,\big(\piix\big(\bbw{\omega}\, \bbw{\omega_2}^{-1}\big)\big)^{-1}\bigg)$.  By Proposition~\ref{prop:lem22bis},   $A_{X_1}  =  \bbw{\omega_2} A_Y \bbw{\omega_2}^{-1}\cap A_X =   \bbw{\omega_2} (A_Y\cap  \bbw{\omega_2}^{-1} A_X \bbw{\omega_2} )\bbw{\omega_2}^{-1} = \bbw{\omega_2} A_{Y_1}\bbw{\omega_2}^{-1}$,  so we get $$\bbw{\omega} A_Y \bbw{\omega}^{-1}\cap A_{X} = \bbw{\omega} \bigg( A_Y \cap  \bigg( \big(\bbw{\omega}^{-1}\piix\big(\bbw{\omega}\, \bbw{\omega_2}^{-1}\big)\bbw{\omega_2}\big) \, A_{Y_1}\, \big(\bbw{\omega}^{-1}\piix\big(\bbw{\omega}\, \bbw{\omega_2}^{-1}\big)\bbw{\omega_2}\big)^{-1}\bigg)\bigg) \bbw{\omega}^{-1}$$

Now,  $\bbw{\omega}^{-1}\piix(\bbw{\omega}\bbw{\omega_2}^{-1})\bbw{\omega_2} = (\bbw{\omega}^{-1}\bbw{\omega_2})\bbw{\omega_2}^{-1}\piix(\bbw{\omega}\bbw{\omega_2}^{-1})\bbw{\omega_2}$ and $\bbw{\omega}^{-1}\bbw{\omega_2}$ lies in $CA_I$,   then,  $\piix(\bbw{\omega}\, \bbw{\omega_2}^{-1})$ lies in $CA_X$ and  $\bbw{\omega}^{-1}\piix(\bbw{\omega}\, \bbw{\omega_2}^{-1})\bbw{\omega_2}$ lies in $CA_I$. 
So,  up to replacing $(X,Y,\bbw{\omega})$ with $(Y,Y_1, \bbw{\omega}^{-1}\piix(\bbw{\omega}\, \bbw{\omega_2}^{-1})\bbw{\omega_2})$,  we may assume  $Y\subseteq X$ with $\bbw{\omega}$ in $CA_I$.  In  this case $\bbw{w_2} = 1$ and $A_{X_1} = A_{Y}$.   Thus $\bbw{\omega} A_Y \bbw{\omega}^{-1}\cap A_{X} =$  $\bbw{\omega} A_Y \bbw{\omega}^{-1} \cap \bigg( \piix\big(\bbw{\omega}\big)A_{Y}\big(\piix\big(\bbw{\omega}\big)\big)^{-1}\bigg) =  \bbw{\omega} \bigg(A_Y  \cap \bigg( (\bbw{\omega}^{-1}  \piix(\bbw{\omega})\big)A_{Y}\big(\bbw{\omega}^{-1} \big(\piix(\bbw{\omega})\big)^{-1}\bigg)\bigg) \bbw{\omega}^{-1}$  and we are done. 
\end{proof}

Acknowledgement: I thanks Fran\c cois Digne and Jean Michel for useful comments. 
\bibliographystyle{acm}
\bibliography{biblio}

\begin{thebibliography}{10}

\bibitem{BlPa2022}
{\sc Blufstein, M., and Paris, L.}
\newblock Parabolic subgroups inside parabolic subgroups of {A}rtin groups.
\newblock arXiv:2204.05142.

\bibitem{Bou1968}
{\sc Bourbaki, N.}
\newblock {\em Groupes et {A}lg\`ebres de {L}ie chapitres 4,5,6}.
\newblock Hermann, 1968.

\bibitem{ChP2014}
{\sc Charney, R., and Paris, L.}
\newblock Convexity of parabolic subgroups in {A}rtin groups.
\newblock {\em Bull. Lond. Math. Soc. 46\/} (2014), 1248--1255.

\bibitem{Cum2019}
{\sc Cumplido, M., Gebhardt, V., Gonzalez-Meneses, J., and Wiest, B.}
\newblock On parabolic subgroups of {A}rtin-{T}its groups of spherical type.
\newblock {\em Advances in Mathematics 352\/} (2019), 572--610.

\bibitem{Cum2020}
{\sc Cumplido, M., Martin, A., and Vaskou, N.}
\newblock Parabolic subgroups of large-type {A}rtin groups.
\newblock Preprint. arXiv:2012.02693.

\bibitem{GoPa2012}
{\sc Godelle, E., and Paris, L.}
\newblock K($\pi$, 1) and word problems for infinite type {A}rtin{-}{T}its
  groups, and applications to virtual braid groups.
\newblock {\em Math. Z. 272\/} (2012), 1339--1364.

\bibitem{Mor2021}
{\sc Morris-Wright, R.}
\newblock Parabolic subgroups in {FC}-type {A}rtin groups.
\newblock {\em Journal of Pure and Applied Algebra 225}, 1 (2021), 106468.

\bibitem{Par2002}
{\sc Paris, L.}
\newblock {A}rtin monoids inject in their groups.
\newblock {\em Comment. Math. Helv. 77\/} (2002), 609--637.

\bibitem{Sol1976}
{\sc Solomon, L.}
\newblock A {M}ackey formula in the group ring of a {C}oxeter group.
\newblock {\em J. Algebra 41\/} (1976), 255--268.

\bibitem{Tit1961}
{\sc Tits, J.}
\newblock Groupe et g{\'e}om{\'e}tries de {C}oxeter.
\newblock {\em mimeographed notes, IHES\/} (1961).

\bibitem{Tit1974}
{\sc Tits, J.}
\newblock {\em Buildings of spherical Type and Finite type BN-pairs}, vol.~386.
\newblock Sringer-Verlag, 1974.

\bibitem{VdL}
{\sc {Van der Lek}, H.}
\newblock {\em The homotopy type of complex hyperplane complements}.
\newblock PhD thesis, Nijmegen, 1993.

\end{thebibliography}
\end{document}